\documentclass[smallextended,natbib,runningheads]{svjour3}
\journalname{Annals of the Institute of Statistical Mathematics}
\smartqed  
\usepackage{graphicx}
\usepackage{amsmath,amssymb,mathrsfs,varioref}
 \usepackage{dsfont}

\begin{document}

\title{Improper   vs finitely additive distributions as limits  of  countably additive  probabilities}
\subtitle{}

\titlerunning{FAP vs improper limits}        

\author{Pierre Druilhet$^*$      \and
      Erwan Saint Loubert Bi\'e
}

\authorrunning{Druilhet, Saint Loubert Bi\'e} 

\institute{Pierre Druilhet \at
             Universit\'e Clermont Auvergne, CNRS, LMBP, UMR 6620, Campus des C\'ezeaux, 3 place Vasarely, TSA 60026, CS 60026, 63178 AUBIERE Cedex, FRANCE \\
              \email{pierre.druilhet@uca.fr}           
           \and
           Erwan Saint Loubert Bi\'e \at
             Universit\'e Clermont Auvergne, CNRS, LMBP, UMR 6620, Campus des C\'ezeaux, 3 place Vasarely, TSA 60026, CS 60026, 63178 AUBIERE Cedex,  FRANCE
}

\date{Received: date / Revised: date}

\maketitle

\begin{abstract}
 
 {
 The Bayesian paradigm with proper priors  can be extended either to improper distributions  or to finitely additive probabilities (FAPs). 
Improper distributions and   diffuse FAPs   can be seen as limits of proper distribution sequences for specific   convergence modes.
In this paper, we compare these two kinds of limits.
We show that improper distributions and FAPs   represent two distinct features of the limit behavior of a sequence of proper distribution. More specifically, an improper distribution characterizes the behavior of the sequence inside the domain, whereas diffuse FAPs characterizes  how the mass concentrates on the boundary of the domain.   Therefore, a diffuse FAP cannot be seen as the counterpart of an improper distribution.
As an illustration, we consider several approach to define  uniform FAP distributions on   natural numbers as an equivalent of improper flat prior.  We also show that expected logarithmic convergence may depend on the chosen sequence of compact sets. }
\end{abstract}

 \section{Introduction}
 
 Improper priors and finitely additive probabilities (FAP) are the  two main extensions of the standard Bayesian paradigm based on  proper priors, i.e. countably additive probabilities \citep[see][p.15]{hart1983}. Both extensions induce  paradoxical phenomena such as strong inconsistency \citep{ston1976,dubi1975} or   marginalization paradoxes \citep{dastzi1973} that do not  occur  with proper priors. To have a better understanding of these phenomena, 
   some authors such as  \citet{stone1982} or  \citet[][p.218]{kada1986}, consider improper distributions and FAPs as limits of  proper prior sequences w.r.t. to appropriate topologies. Heuristically, this approach seems to establish a link between improper distributions and FAP. .
  
Seeing a FAP as a limit  is a way to preserve the total mass   equal to 1, while sacrificing the countable additivity. This point of view has been mainly supported by \citet{defi1972}. On the other hand, improper distributions aim at preserving the countable additivity, while  sacrificing  a total mass equal to 1. Improper distributions appear naturally in the framework of conditional probability, see \citet{reny1955} 	and more recently \citet{taralind2010,taralind2016} and \citet{lindtara2018}. Conditional probability spaces are also related to projective spaces of measures \citep{reny1970} which have a natural quotient space topology  and a natural convergence mode, named $q$-vague convergence by \citet{bidr2016}. Bayesian inference with  improper posterior is justified by  \citet{tatuli2018} from a theoretical point of view.
  \citet{ bobidr2018} consider the convergence of proper distribution sequences  to an improper posterior for Bayesian   estimation of abundance   by removal sampling.  \citet*{tulari2018} propose to adapt MCMC for the  estimation of  improper posteriors. In an other approach,  \citet{akai1980} consider the convergence of 
  posterior distributions w.r.t. an entropy criterion when the posterior distributions are proper.

 
In this paper, we mainly   consider convergence   of prior distributions to FAPs or to improper distributions, regardless to any statistical model.  In Section \ref{section.defcv}, we define the notion of limits in the settings of improper distributions and of FAPs. We show that improper distributions and FAPs  represent two distinct characteristics of a sequence of proper distributions. Therefore, they cannot be connected by the mean of proper distribution sequences.  In Section \ref{section.uniform}, we revisit the notion of uniform distribution on integers in the light of our results. In Section \ref{section.examples}, we illustrate with some examples the fundamental difference between convergence to an improper prior and to a  FAP. In Section \ref{section.elc}, we consider expected logarithmic convergence, defined by \cite{bebesu2009} to approximate an improper distribution by a sequence of truncated proper priors.  We apply some of the methods used in Section   \ref{section.defcv} to propose an example where this convergence mode depends on the chosen  sequence of compact sets.

 \section{Convergence  of probability  sequences}
 \label{section.defcv}
 We denote by $\mathcal C_b$ the set of continuous real-valued  bounded functions on a space  $\Theta$ and  by $\mathcal C_K$ the set of continuous real-valued functions with compact support.
 For a $\sigma$-finite measure $\pi$, we denote $\pi(f)=\int f(\theta) \;d\pi(\theta)$.
 Let  $\{\pi_n\}_{n\in \mathbb N}$ be a sequence of proper distributions.  The usual   converge mode of  $\{\pi_n\}_{n\in \mathbb N}$  to a proper prior $\pi$ is the narrow   convergence,  also called weak convergence or convergence in law, defined by:
 \begin{equation}
 \label{eq.cvnarrow}
 \pi_n \xrightarrow[n\rightarrow +\infty]{narrowly} \pi\iff \pi_n(f) \xrightarrow[n\rightarrow +\infty]{}\pi(f)\quad \forall f\in\mathcal C_b \; .
 \end{equation}
 
 When it exists, the narrow limit of $\{\pi_n\}_n$ is necessarily unique. In this section, we consider two alternative convergence modes when there is no narrow limit, and especially    when the total  mass tends to concentrate around  the boundary on the domain, more precisely when $\lim_n \pi_n(f)=0$ for all $f$ in $\mathcal C_K$. The idea is to consider a proper prior either as a special case of FAP or as a special case of  a Radon measure, and for each case, to define  a convergence mode in a formalized way.  
 
 
 ~

 In the following,   $\Theta$ is a  locally compact separable metric  space.   This is the case, for example, for usual topological finite-dimensional  vector spaces or for denumerable sets with the discrete topology.  In the latter case, any  function is continuous and a compact set is a finite set.  
 
 \subsection{Convergence to an improper distribution}
 To extend the notion of   narrow limits, we consider here  proper distributions  within the set of projective space of positive Radon measures as follows: we denote by $\mathcal R$ the set of non-null Radon measures, that is     regular countably additive measures with finite mass  on each compact set.  Note that, in the discrete case,  any $\sigma$-finite measure is a Radon measure. 
 
 ~

 We define an improper distribution as an unbounded  Radon measure which appears  in parametric Bayesian statistics  \citep[see, e.g.][]{jeff1983}.
 The  \textit{projective space} $\overline{\mathcal R}$
 associated to $\mathcal R$ is the quotient space for the equivalence relation $\sim$  defined by $\pi_1\sim\pi_2$ iff $\pi_2=\alpha\, \pi_1$ for some positive scalar factor $\alpha$. To each Radon measure $\pi$   is associated a unique equivalence class $\overline\pi=\{\pi'=\alpha\,\pi\,;\,\alpha>0\}$. Therefore, a projective space is a space where  objects are defined up to a positive scalar factor. It is natural in Bayesian statistics  to consider such projective space since   two equivalent priors give the same posterior.    The projective space $\overline{\mathcal R}$ is also naturally linked with  conditional probability spaces     \citep
 {reny1955}. 
 All the results presented below on the convergence mode w.r.t. to the projective space $\overline{\mathcal R}$ can be found in \citet{bidr2016}. The usual topology on $\mathcal R$ is the vague topology defined by 
 \begin{equation}
 \label{eq.cvvague}
 \pi_n \xrightarrow[n\rightarrow +\infty]{vaguely} \pi\iff \pi_n(f) \xrightarrow[n\rightarrow +\infty]{}\pi(f)\quad \forall f\in\mathcal C_K\ .
 \end{equation}

 From  the related  quotient topology, we can derive a convergence mode,  called q-vague convergence: a sequence  $\{\pi_n\}_n$ in $\mathcal R$ converges $q-vaguely$ to a (non-null) improper distribution $\pi$ in $\mathcal R$  if $\overline \pi_n$ converges to $\overline\pi$ w.r.t.  the quotient topology where $\overline \pi_n$ and $\overline \pi$ are the equivalence classes associated to $\pi_n$ and  $ \pi$. The  limit $\overline \pi$  is unique whereas $\pi$ is unique only  up to   a positive scalar factor. It is not always tractable  to   check a convergence in the  quotient space. However, there is  an equivalent definition in the initial space $\mathcal R$:  $\{\pi_n\}_n$   converges $q-vaguely$ to  $\pi$ if there exists some scalar factors $\alpha_n$ such that $\{\alpha_n\,\pi_n\}_n$ converges vaguely to $\pi$:
 \begin{equation}
 \label{eq.cvqvague2}
 \pi_n \xrightarrow[n\rightarrow +\infty]{q-vaguely} \pi\iff a_n\pi_n \xrightarrow[n\rightarrow +\infty]{vaguely}\pi\quad   \textrm{for some } a_1,a_2,...>0\, .
 \end{equation}

 The $q$-vague convergence can be considered as an extension of the narrow convergence in the sense that if $\{\pi_n\}_n$ and  $\pi$  are proper distributions and  $\{\pi_n\}_n$ converges narrowly to $\pi$ then   $\{\pi_n\}_n$ converges q-vaguely to $\pi$. Note that the converse part holds if and only if $\{\pi_n\}_n$ is tight \citep[see][Proposition 2.8]{bidr2016}.
 
 ~

 When a sequence $\{\pi_n\}_n$ of proper distributions converges q-vaguely to an improper distribution, then  $\lim_n\pi_n(K)=0$ for any compact $K$ \citep[Prop. 2.11]{bidr2016}.  The following lemma gives an apparently stronger, but in fact equivalent, result. It will be useful to establish our main result and to construct examples in Section \ref{section.exampleFAPvsqvague} and 5.
 
 \begin{lemma}\label{lemma.suitecpctslow}
 	Let $\{\pi_n\}_n$ be a sequence of proper distributions such that $\lim_n\pi_n(K)=0$ for any compact $K$. Then there exists a non-decreasing  sequence of compact sets $K_n$ such that $\cup_n K_n=\Theta$ and  $\lim_n\pi_n(K_n)=0$. Moreover, $K_n$ may be chosen such that, for any compact $K$, there exists an integer $N$ such that $K\subset K_N$.
 \end{lemma}
 \begin{proof}
 	Let $\widetilde K_m$, $m\geq 1$, be an increasing sequence of compact sets with $\cup_m \widetilde  K_m=\Theta$. For each $m $, $\lim_n\pi_n(\widetilde K_m)=0$, so  there exists an integer $N_m$ such that $N_m>N_{m-1}$ and  $\pi_n(\widetilde K_m)\leq 1/m$ for $n>N_m$. Consider now such a sequence of integers $N_m$, $m\geq 1$. For any $n$ there exists a unique integer  $m$ such that  $N_m\leq n <N_{m+1}$. We define $K_n$ by $K_n=\widetilde K_m$. So,  $\pi_n(K_n)=\pi_n(\widetilde K_{m})\leq 1/m$. Since $m$ increases with $n$, $\lim_n\pi_n(K_n)=0$. Furthermore, the sequence $\widetilde K_m$ can be chosen such that, for any compact $K$, $K$ is a subset of all but finitely many $\widetilde K_m$, \citep[see  e.g.][Lemma 29.8]{bauer2001}. By construction, the same property holds for  the sequence $K_n$. \qed
 \end{proof}
 
 Note that,  $\lim_n\pi_n(K)=0$ for any compact set $K$ does not imply that $\{\pi_n\}_n$ converge q-vaguely, as shown in Section \ref{section.examples} on some examples.
 \subsection{Convergence   to a FAP}
 \label{section.FAPlimit}

Here, we consider proper distributions as special cases of FAPs.
Denote by $\mathcal F_b$  the set of bounded real-valued measurable  functions on $\Theta$.  A FAP $\pi $ is a linear functional on $\mathcal F_b$  which is positive, { i.e.} $\pi(f)\geq 0$ if $f\geq 0$, and which satisfies  $\pi(1)=1$. Therefore, the set of FAPs is included in the topological dual of $\mathcal F_b$ 
 equipped with the sup-norm. For any measurable set $E\subset \Theta$, we define  $\pi(E)=\pi(\mathds 1_E)$, where $\mathds 1_E(x)=1$ if $x\in E$ and 0 otherwise. We also denote $\int f(\theta) \;d\pi(\theta)=\pi(f)$.

For most authors \citep[see, e.g.][]{heathsudd1978}, a FAP is a linear functional on the set of bounded real valued functions. Here, we do impose a  measurability condition, since we require   proper distributions to be special cases of FAPs. 
 In the case where  $\Theta$ is a denumerable set equipped with the usual discrete topology, any function or set is measurable, and so, both definitions of FAPs are equivalent. 
 
 ~

 Let $\{\pi_n\}_n$ be a sequence of FAPs.
The usual convergence mode for FAPs is associated to the weak* topology: a sequence $\{\pi_n\}_n$ converges to $\pi$ if $\lim_n \pi_n(f) =\pi(f)$ for any $f\in \mathcal F_b$. 

When $\{\pi_n\}_n$ does not converge, we may consider limit points, as proposed by \citet{stone1982} for denumerable sets and extended here to more general sets. The existence     of   limit points relies on the  Banach-Alaoglu-Bourbaki theorem   \citep[see, e.g.][]{rudi1991}, since   a FAP belongs to the unit ball in  the dual of $\mathcal F_b$, which  is compact for the weak$^*$-topology. Hence, for any  sequence $\{ \pi_n\}_n$ of  FAPs,  there exists at least one limit point $\pi$ which is defined as a FAP limit.  We recall that $\pi$ is a limit point of $\{\pi_n\}_n$ for the weak$^*$-topology if and only if for any   integer  $p$, any  $f_1,...f_p$ in $\mathcal F_b$ and any $\varepsilon>0$, there exists  an  infinite number of  $n$ such 
that $\left|\pi_n(f_i)-\pi(f_i)\right|\leq \varepsilon$, $i=1,...,p$. Since $\mathcal F_b$ is not in general first-countable, there does not necessarily exist a subsequence $\{\pi_{n_k}\}_k$ that converges to $\pi$.  We can only say that, for any $f_1,..., f_p$ in $\mathcal F_b$, there exists a subsequence $\{\pi_{n_k}\}_k$ such that $(\pi_{n_k}(f_1),...,\pi_{n_k}(f_p))$ converges to $(\pi(f_1),...,\pi(f_p))$.

 ~

 When $\pi_n$ and $\pi$ are proper distributions, then 
 $\{\pi_n\}_n$ converging narrowly to  $\pi$ does not imply that   $\pi$ is a FAP limit point of $\{\pi_n\}_n$.  Therefore, unlike $q$-vague limits, FAP limit points  cannot be considered as an extension of the narrow convergence.  
 For example,  consider the proper distributions   $\pi_n=\delta_{\sqrt 2/n}$, where $\delta$ is the Dirac measure. The sequence $\{\pi_n\}_n$ converges narrowly to $\pi=\delta_0$ but $\pi$ is not a FAP limit point  of $\{\pi_n\}_n$. To show this, consider  $f(\theta)=\mathds 1_{\mathbb Q} (\theta)\in\mathcal F_b$, with $\mathbb Q$ the set of rational numbers, we have $\lim_n\pi_n(f)=0\neq \pi(f)=1$. We can only say that any FAP limit point
  of the sequence $\{\pi_n\}_n$ will coincide with $\pi=\delta_0$ on the set ${\cal C}_b$. 
 
 To consider a FAP limit  as an extension of the narrow  convergence, we should have defined FAPs on the space $\mathcal C _b $ rather than $\mathcal F _b$. However, with this choice, $\pi (E)$ is not well defined for all measurable sets  $E$.

In the special case where $\Theta$ is a denumerable set, any real-valued function on $\Theta$ is continuous. 
 So, if  a sequence of proper distributions $\{\pi_n\}_n$ converges narrowly to  a proper distribution $\pi$, then $\pi$ is a FAP limit point.

 ~

Another way to extend the notion of limit can be obtained   by using the Hahn-Banach theorem as follows \citep[see][]{huis2016}:
let $\mathcal S_c$ be the set of $f\in\mathcal F_b$ such that  $\lim_n \pi_n(f)$ exists. 
 A FAP $\pi$ is said to be an extended  FAP limit of $\{\pi_n\}_n$ if
 $\lim_n \pi_n(f) =\pi(f)$ for any $f\in \mathcal S_c$,
 and if $\pi(f) \leq \limsup_{n} \pi_n(f)$.
 The existence of  a FAP $\pi$ satisfying this requirement is guaranteed by the Hahn-Banach theorem \citep[see][]{rudi1991}: define the linear function $\Phi$ on $\mathcal S_c$ by $\Phi(f)=\lim_n \pi_n(f)$ and the sublinear functional $p(f)=\limsup_n \pi_n(f)$. Then, there exists a linear functional $\pi$ on $\mathcal F_b$ that coincides with $\Phi$ on $\mathcal S_c$ and that satisfies  $ \pi(f)\leq p(f)$ on $\mathcal F_b$. 
 The condition $\pi(f)\leq p(f)$ implies that    $\pi$ is a FAP. Conversely, an extended  FAP limit necessarily  satisfies $\pi(f)\leq p(f)$. Replacing    $f$ by  $-f$ gives  $\pi(f)\geq \liminf_n\pi_n(f)$.  Therefore, an extended  FAP limit can be characterized by the  following lemma:

 \begin{lemma}
 	\label{lemma.cvFAP}
 	A FAP $\pi$ is an extended FAP-limit of the sequence $\{\pi_n\}_n$ if and only if   $\textrm{for any } f\in\mathcal F_b $ 
 	\begin{equation}\label{eq.cvFAP}
 	\liminf_n\pi_n(f) \leq \pi(f)\leq  \limsup_n\pi_n(f) 
 	\end{equation}
 	or equivalently   if and only if $ \textrm{for any } \textrm{ measurable set } E$
 	\begin{equation}
 	\label{eq.cvFAPdiscret}
 	\liminf_n\pi_n(E) \leq \pi(E)\leq  \limsup_n\pi_n(E)  .
 	\end{equation}
 	
 \end{lemma}
Note that the sequence $\{\pi_n\}_n$ converges to $\pi$ for the  weak$^*$ topology if and only if $\mathcal S_c=\mathcal F_b$.
In general,  an extended FAP limit is not unique  and its existence relies on the axiom of choice.

~

~

The  set of limit points  of  $\{\pi_n\}_n$   is included in the set of extended  FAP limits.  The converse inclusion is false in general. Inequalities (\ref{eq.cvFAP}) or (\ref{eq.cvFAPdiscret})  hold for limit points but are not sufficient  to characterize them. It is easy to see that  the closed convex hull of the set of limit points is included in the set of extended  FAP limits.  We conjecture that, conversely, the set of extended  FAP limits defined by (\ref{eq.cvFAP}) is the closed convex hull of the set of limit points.
As a simple example, consider the sequence $\{\pi_n\}_n$ with $\pi_{2n}=\delta_0$ and $\pi_{2n+1}=\delta_1$. There are only two  limit points $\delta_0$ and $\delta_1$, whereas any $\pi=\alpha \delta_0+(1-\alpha) \delta_1$, $0\leq\alpha\leq 1$ is a extended  FAP limit. In Section \ref{section.limdeltan}, we illustrate the difference between these two constructions of limits with another example.

 ~
 
 Even if the notion of FAP limit points is more restrictive than the notion of  extended FAP limit,  the main results, especially   Theorem \ref{thm.sameqvnotFAP}, Corollary \ref{corol.main},  Proposition \ref{prop.PoissonlimitFAP}, Lemma \ref{lemma.FAPlimitcombconvexe} and \ref{lemma.FAPKNslow}   hold for  both of them. In the following, we consider only FAP limit points.

 
 ~

 \subsection{FAP limit points  vs q-vague convergence}
 \label{section.FAPvsqvague}
 The fact that a sequence of proper distributions has both  improper and  FAP limit points may suggest  a connection between the two notions as proposed heuristically by several authors, such as \citet{levi1980,stone1982,kada1986}. The following results show  that this is not the case. Roughly speaking, it is shown that any  FAP which is a   limit point of some proper distribution sequence  can be connected to  any improper prior   by this mean.

 \begin{theorem}
 	\label{thm.sameqvnotFAP}
 	Let  $\{\pi_n\}_n$ be a  sequence of proper distributions such that  $\lim_{n}\pi_n(K)=0 $ for any compact set $K$. Then, for any improper distribution $\pi$,  it can be constructed a sequence $\{\widetilde\pi_n\}_n $ which  converges q-vaguely to $\pi$ and which has the same set of  FAP limit points  as      $\{\pi_n\}_n$.
 \end{theorem}
 
 \begin{proof}
 	For any FAP or any  proper or improper distribution $\mu$, we define the distribution  $(\mathds{1}_A\, \mu)$ by $(\mathds{1}_A\,\mu)(f)=\mu(\mathds 1_A \,f)$ where $A$ is any measurable set.	From  Lemma \ref{lemma.suitecpctslow}, it can be constructed  an  exhaustive increasing  sequence  $K_n$  of compact sets such that $\lim_n \pi_n(K_n)=0$.  
 	Put    $\gamma_n=\pi_n(K_n)$  and define the sequence of proper distributions $\widetilde \pi_n=\gamma_n \frac{1}{\pi(K_n)} \mathds 1_{K_n}\pi$ $+$ $  (1-\gamma_n)  \frac{1}{\pi_n  (K_n^c)}\mathds{1}_{K_n^c}\,\pi_n $, with  $K^c$  the complement of $K$. By Lemma    \ref{lemma.FAPKNslow} and  \ref{lemma.FAPlimitcombconvexe} in  Appendix A, $\widetilde\pi_n$ has the same FAP limit points    as $\{\pi_n\}$. By Lemma \ref{lemma.cvqvtronque},  $\widetilde\pi_n$ converges q-vaguely to  $\pi$. \qed
 	
 \end{proof}

 \begin{corollary}
 	\label{corol.main}
 	Let $\{\pi_n\}_n$ be a sequence of proper distributions that converges q-vaguely to 
 	an improper distribution $\pi^{(1)}$. Then, for any other  improper distribution $ \pi^{(2)}$, it can be constructed a sequence $\{ \widetilde\pi_n\}_n$ that converges q-vaguely to $\pi^{(2)}$ and that has the same FAP limit points as $\{\pi_n\}_n$.
 \end{corollary}
 
 We have shown that no direct link between an improper limit and a FAP limit point can be established.
One can only say that, if a sequence of proper distributions converges to an improper distribution, its FAP limit point $\pi$  are {\em diffuse}, i.e. $\pi(K)=0$ for any compact set $K$. 
 
 \section{Uniform distribution on integers} 
 \label{section.uniform}
 In this section, we compare different  notions of uniform distributions on the set $\mathbb N$ of integers, by using several considerations such as limit of proper uniform distributions.

 We also illustrate the fact that FAP uniform distributions  are not  well  defined objects \citep[][pp.122,224]{defi1972}. Contrary to  uniform improper distributions, FAP limit points of  uniform distributions    on an  exhaustive sequence of compact sets are highly dependent on the choice of that sequence.

 \subsection{Uniform improper distribution}
 There are several equivalent  ways  to define a uniform improper  prior on integers. These definitions lead  to a unique, up to a scalar factor,  distribution.  The uniform distribution can be defined directly as a  flat distribution, i.e. $\pi(k)\propto 1$ for any integer $k$. It is  the unique (up to a scalar factor) measure that is shift invariant, i.e. such that  $\pi(k+A)=\pi(A)$ for any integer $k$ and any set of integers $A$. The uniform distribution is also the q-vague  limit  of the sequence of  uniform proper  distributions on $K_n=\{0,1,...,n\}$. More  generally and equivalently, the uniform distribution is the      q-vague limit of any sequence of proper uniform  priors on an exhaustive increasing sequence $\{K_n\}_n$ of finite subsets of integers. 
 
 \subsection{Uniform finitely additive probability}
 The notion of uniform finitely additive probabilities is more complex. Contrary to the improper   case,  there is no explicit definition    since   $\pi({k})=0$ for any integer $k$. We present here several non equivalent approaches to define a uniform FAP. The first two ones can be found in  \cite{kadaohagan1995} and   \citet{schikada2007}.
 
 \subsubsection{Shift invariant  (SI) uniform distribution}
 
 As for the improper case, a uniform FAP $\pi$  can be defined as being any shift invariant FAP, i.e. a FAP satisfying  $\pi (A) = \pi (A+k)$ for any subset of integers  $A$ and any integer $k$. Such a distribution will  be called SI-uniform. In that case, one necessarily has : $\pi(k_1+k_2\times \mathbb N)=k_2^{-1}$, for any $(k_1,k_2) \in \mathbb{N}\times\mathbb{N}^*$.  In \cite{kadaohagan1995}, the authors investigate the properties of FAPs  satisfying only  $\pi(k_1+k_2\times \mathbb N)=k_2^{-1}$, where the sets   $k_1+k_2\times \mathbb N$ are called \emph{residue classes}.

 \subsubsection{Limiting relative frequency (LRF) uniform distributions.}
 
 \cite{kadaohagan1995} consider a stronger condition to define uniformity. For a  subset  $A$,  define its Limiting Relative Frequency LRF$(A)$  by  $$\textrm{LRF}(A) = \lim_{N\to\infty} \frac{\#\{k\leq N, \; \text{s.t.}\; k\in A\}}{N+1}, $$
 when this limit exists. A  FAP $\pi$ on $\mathbb N$ is said to be {\em LRF uniform} if  $\pi (A) = p$  when $\textrm{LRF}(A)=p$. 
 
 Let  $\pi_n$ be the  uniform proper  distribution on  $K_n=\{0,1,...,n\}$, then $\textrm{LRF}(A) =$   $ \lim_{n\to\infty} \pi_n(A)$. Therefore, any FAP   limit point of $\pi_n$ is LRF uniform.  
 In fact,   a FAP $\pi$ is LRF uniform if and only if it is an extended FAP limit of  $\{\pi_n\}_n$.
 
  It is worth noting that, unlike the $q$-vague limit,
   FAP limit points are highly dependent on the  choice of the increasing exhaustive  sequence of finite sets $K_n$. Changing the sequence $\{K_n\}_n$ changes the notion of uniformity. For example, if $\widetilde\pi_n$ is the uniform distribution on  $K_n=\{2k\,;\; 0\leq k\leq n^2\} \cup \{2k+1\,;\;0\leq k\leq n\}$, then $\lim_n \widetilde \pi_n(2\mathbb N)= 1$, whereas $\lim_n\pi_n(2 \mathbb N)=1/2$.

 \subsubsection{Bernoulli Scheme (BS) uniform distribution}

 We propose here another notion of uniformity that is not dependent of the choice a particular increasing sequence of finite sets $K_n$ as for the LRF uniformity. Consider a Bernoulli Scheme, that is, a sequence    $\{X_k\}_{k\in\mathbb N}$ of i.i.d.  Bernoulli distributed  random variables with mean $p\in[0,1]$. Define the random set $A(X)=\{k \in \mathbb N, \text{ s.t. } X_k =1 \}$. A   FAP $\pi$ is said to be  {\em BS-uniform } if, for any $p\in[0;1]$, $\pi(A(X)) = p$, almost surely. By  the strong law of large numbers, LRF$(A(X))=p$, almost surely.
 
 \begin{proposition} 
 	Let $\{K_n\}$ be an    increasing sequence of finite subsets of $\mathbb{N}$, with $\cup_{n\in\mathbb N} K_n $ being infinite. Then any FAP  which is a limit point   of the sequence $\pi_n$ of uniform distributions on $K_n$ is BS-uniform.
 \end{proposition} 
 When $\cup_{n\in\mathbb N} K_n =\mathbb N$, this proposition shows that any FAP limit point of uniform distribution sequences is BS-uniform. In particular, a LRF uniform FAP is also BS uniform.    However, if for example $K_n$ is  the set of even numbers less or equal to $n$, then any FAP  limit point $\pi$  of the sequence of  uniform distributions on $K_n$ is BS-uniform. However, $\pi$ is not uniform on $\mathbb N$ but uniform on $2\mathbb N$. Therefore, BS uniformity looks much more like a necessary condition for a FAP to be uniform, than like a complete definition.

 \section{Comparison of convergence modes on examples}
 \label{section.examples}
 
 We consider here some examples that illustrate the difference between convergence of proper distributions to an improper distribution and to a FAP.

 \subsection{FAP limit points  on ${\mathbb N}$.}
 \label{section.limdeltan}
  For a sequence $\{\pi_n\}_n$ of proper distributions on $\mathbb N$, it is known that there does not necessarily exist a q-vague limit,  but if it exists, it is unique up to a scalar factor, i.e. it is unique in the projective space of Radon measures. At the opposite, we have seen that a FAP limit point always exists but is not necessarily unique. 
 
 We illustrate the non-uniqueness of FAP limit point   with  an extreme case. Consider the  sequence of proper distributions $\pi_n = \delta_n$, where $\delta _n$ is  the Dirac measure on $n$. This sequence has no q-vague limit since $\pi_n(k)=0$ for $n>k$.

 Let $\pi$ be a  FAP limit points.  For any subset $A$, there exists a subsequence $\{\pi_{n_k}\}$ such that  $\pi_{n_k}(A)$  convergences to $\pi(A)$.  So, $\pi(A)\in\{0,1\}.$ Therefore  $\pi$ is any  \textit{remote} FAP, that is a diffuse  FAP   such that $\pi(A)\in\{0,1\}$, as defined by \citet[][p. 92]{dubi1975}.  This also proves the existence of remote FAPs.   Note that a remote FAP is neither BS uniform nor SI  and therefore cannot be LRF uniform. As a remark, the extended FAP limits of $\pi_n$ are   all the diffuse FAPs.

 \subsection{Convergence  of sequence of  Poisson distributions}
 \label{section.poisson}

 We consider here the sequence $\{\pi_n\}_n$ of  Poisson  distributions with mean  $n$. Although $\lim_n\pi_n(K)=0$ for any finite set $K$, this sequence of proper distributions   does not converge   $q$-vaguely  to any improper distribution  \citep[][\S~5.2]{bidr2016}. As a remark, let $\widetilde\pi_n$ be the shifted  measures  defined on the set of positive and integers integers $\mathbb Z$ by $\widetilde\pi_n( B) = \pi_n(B+n)$, where 
 $\pi_n$ can be seen as a measure on the set $\mathbb Z$, with $\pi_n(k)=0$ for $k<0$. Then, using the approximation of the Poisson distribution by an normal distribution, it can be shown that the sequence $\widetilde\pi_n$  converges  $q-$vaguely to the improper uniform measure on the set $\mathbb Z$.  
 
 We consider now the  FAP limit point  of the sequence    $\{\pi_n\}_n$. The next result shows that these limits have some properties of uniformity described in Section \ref{section.uniform} but not all of them. The proof is  given in Appendix B.

 \begin{proposition}
 	\label{prop.PoissonlimitFAP}
 	Any FAP   $\pi$ which is a limit point  of the sequence $\{\pi_n\}_n$ of Poisson distribution with mean $n$  is SI-uniform and BS-uniform but not necessarily  LRF-uniform. 
 \end{proposition}
 
 Therefore, the FAP limit points   of the  Poisson distribution sequence are examples of SI- and BS-uniform distributions that are not LRF uniform.
 \citet{kaji2014} give
 another  example of SI but not LRF uniform FAPs using paths of random walks. Even if they consider FAPs on a subset of  bounded functions, it can be extended to $\mathcal F_b$ by using the  Hahn-Banach theorem similarly to Section \ref{section.FAPlimit}.

 \subsection{FAP vs q-vague convergence of uniform proper  distributions}
 \label{section.exampleFAPvsqvague}

 To illustrate the fact that any FAP limit point can be related with any improper distribution,  consider again the sequence  $\{\pi_n\}_n$ of Poisson distributions with mean $n$  and let $\pi_0$ be any  improper distribution on the integers.  Since  $\lim_n\pi_n(K)=0$ for any finite set, from Lemma \ref{lemma.suitecpctslow} we can construct an exhaustive sequence of finite set $K_n$ such that $\lim_n \pi_n(K_n)=0$.  Put  $K_n = \{k\in \mathbb N, \;  k \leq n/2 \}$, which satisfies this condition. Define the sequence of proper distributions $\widetilde \pi_n$ by :

 \begin{equation}
 \label{eq.seqpoisexample}
 \widetilde\pi_n(A)=\pi_n(K_n) \frac{\pi_0(A\cap K_n)}{\pi_0(K_n)} + (1-\pi_n (K_n)) \frac{\pi_n(A\cap K_n^c)}{\pi_n(K_n^c)} \vspace{1mm}
 \end{equation}
 for any set $A$.  From Theorem \ref{thm.sameqvnotFAP}, $\{ \widetilde{\pi}_n\}_n$ converges $q$-vaguely to $\pi_0$ and has the same FAP limit points as $\{\pi_n\}_n$.

 As another example, let $\{\pi_n\}_n$ be the sequence  of uniform distributions on $\{0,1,...,n\}$ and choose  $K_n = \{k\in \mathbb N, \;  k \leq \sqrt n \}$. Then, $\lim_n \pi_n(K_n)=0$. Therefore, for any improper distribution  $\pi_0$ on the set of integers, the sequence constructed as in  (\ref{eq.seqpoisexample}) has the same FAP limit points as those of the sequence of  uniform distributions $\{\pi_n\}_n$ and  converges q-vaguely to $\pi_0$. This shows again the difficulty to connect the uniform improper distribution and uniform FAPs by limits of proper distributions.
 
 \subsection{Convergence of beta distributions}
 In this section, we consider the limit of the sequence of   Beta distribution $\pi_{n}$ $=$ $\text{Beta}(a_n,b_n)$ defined on $\Theta=]0,1[$ when $a_n$ and $b_n $ go to 0.  We will   see that, contrary to the improper limit,  the FAP limit points depend on the way $a_n$ and $b_n$ go to $0$. 
 This   illustrates again  the difference between the two  kinds of    limits.  
 
 The density of a beta distribution Beta$(a,b)$ is given by  
 $$
 \pi_{a,b} (x) =\frac{1}{\beta(a, b)} \, x^{a-1} (1-x)^{b-1}  \text{ for  } \; x \in ]0;1[
 $$
 where $\beta(a,b)$ is the beta function.
 
 From \citet{bidr2016}, the unique (up to a scalar factor)  q-vague limit of  $\text{Beta} (a_n,b_n)$ when  $a_n$ and $b_n $ go to 0 is  the Haldane improper distribution: $$\pi_H(x)=\frac{1}{x(1-x)}  \text{ for  } \; x \in ]0;1[.$$

 The q-vague  limit gives no information on   the relative concentration of the mass around $0$ and $1$: for $0<u<v<1$,   $\pi_H(]0,u])=\pi_H(]v,1[])=+\infty$. To explore this concentration, we   temporarily replace  the space $\Theta$ by $\overline \Theta=[0,1]$. This has no consequence on the Beta distributions but the Haldane  distribution is no longer a q-vague limit of the sequence. Put $c_n=a_n/b_n$ and assume that $\{c_n\}_n$ converges to some $c\in[0,1]$. The sequence  $\{\pi_{n}\}_n$ converges narrowly, and hence q-vaguely, to the proper distribution $\widetilde\pi=\frac{1}{1+c}\delta_0+\frac{c}{1+c}\delta_1$. Contrary to the Haldane prior, $\widetilde{\pi}$ shows how the mass concentrates on the boundary of the domain, but gives no information on the behavior of the sequence inside the domain. Note that the Haldane distribution is not a Radon measure on $\overline{\Theta}$ since  $\pi_H([0,1])=+\infty$ where $[0,1]$ is a compact set. Therefore $\pi_H$ cannot be a candidate for the q-vague limit on $\overline{\Theta}$.
 
 We now consider the  FAP limit points on $\Theta=]0,1[$ of $\pi_{n}$ and we show that they give an information similar  to that given by $\widetilde{\pi}$ on the way the mass concentrate on the boundary of the domain.  Again, we assume that $c_n=a_n/b_n$ converges to some   $c\in [0,1]$. Easy calculations show that for any  $0<\varepsilon<1$ $\lim_n \pi_{n}(]0,\varepsilon [) = \frac{1}{1+c}$ and $\lim_n \pi_{n}(]1-\varepsilon,1 [) = \frac{c}{1+c}$.
 Therefore, for any FAP limit point  $\pi$  and for any $\varepsilon \in ]0,1[$, we have  $\pi (]0,\varepsilon [) = \frac{1}{1+c}$ and $\pi (]1-\varepsilon ,1[) =\frac{c}{1+c}$, with $\pi([u,v])=0$ for $0<u<v<1$.

 \section{Expected logarithmic convergence}
 
 \label{section.elc}
 
 In Bayesian statistics,  consider a statistical model $ p(x\vert \theta)$ and an improper prior $\pi(\theta)$ on $\Theta$.
Define the truncated proper prior $\pi_n(\theta)\propto \pi(\theta) \; \mathds{1}_{\theta\in K_n}$, for some  exhaustive increasing sequence of compact sets $\{K_n\}_n$.
 From \cite{bebesu2009}, a sequence of posteriors distributions $\pi_n (\theta\vert x)$ is said to be expected logarithmically convergent to $\pi(\theta\vert x)$ if 
 $$
 \lim_{i\to\infty}
 \int_ {{\cal X}} p_i(x) \kappa(\pi(\cdot\vert x),\pi_i(\cdot \vert x)) dx 
 =
 0\, , 
 $$
 where $p_i(x) = \int_{\Theta} p(x\vert \theta) \pi_i (\theta) d\theta$, and $\kappa(m_1, m_2)$ denotes the Kullback-Leibler distance between probability measures $m_1$ and $m_2$. The prior $\pi$ is said to be permissible w.r.t. $p(x\vert \theta)$ if $\pi(\theta|x)$ is proper and if there exists some exhaustive sequence of compact sets $\{K_n\}_n$ such that $\pi_n(\theta|x)$ is expected logarithmically   convergent to $\pi(\theta|x)$.  Note that $\pi_n(\theta|x)$ converges q-vaguely to $\pi(\theta|x)$, provided that $p(x\vert\theta)$ is continuous w.r.t. $\theta$ for any $x$ \citep[][Prop. 3.1]{bidr2016}. 
 
 An open problem  is to know whether this property is always independent from the choice of the sequence of compact sets $\{K_n\}_n$. We present here a situation  where it is not.
 
 
 The construction of this counter-example relies on the fact that  the tail behavior of a sequence of distributions is not directly related to its $q$-vague convergence as explained in Section \ref{section.FAPvsqvague}. 
  
Consider the following model: for any integers $x$ and $\theta$   (included negative integers) define:
$$ p(x \vert \theta ) = 
\left\{ 
\begin{array}{cc}
\frac{1}{3} & \text{ if }  \theta \geq 1, \; 
x \in \{\left[\frac{\theta}{2}\right], \, 2 \theta, \, 2 \theta +1\} \\[.5mm]
1 & \text{ if }  \theta \leq 0, x=\theta \\[.5mm]
0 & \text{ otherwise}
\end{array}
\right.
$$
where $[l]$  the integer part of $l$, with the particular case $[1/2]=1$. 
 For $\theta\leq 0$, $x\leq 0$,
 we have a deterministic model. Remark also
 that, for $x,\theta\geq 1$, we have the equivalence~:
 $$\textstyle
 \left( \;  x \in \left\{\left[\frac{\theta}{2}\right], \, 2 \theta, \, 2 \theta +1\right\}  \; \right)
 \Leftrightarrow
 \left( \; \theta \in \left\{\left[\frac{x}{2}\right], \, 2 x, \, 2 x +1\right\} \; \right).
 $$
 Consider the flat prior $\pi(\theta)\propto 1$. If we consider only $x,\theta\geq 1$, this model corresponds to a model  proposed by \cite{frasermonetteng1985} and  used by  \cite{bebesu2009} to illustrate their approach.

 Let $K_{n}= \{-a_n \leq \theta \leq b_n\}$ be an  exhaustive sequence of compact sets, with $a_n,b_n \rightarrow +\infty$.  
Denote $I_{n} = \int_{{{\cal X}}} p_{n}(x) \kappa(\pi(\cdot\vert x),\pi_{n}(\cdot \vert x)) dx$. We have:
$$
I_{n} \; = \; \frac{1}{a_n+b_n+1} 
\left( \ln(3) \frac{\big(2b_n+1-\left[ \frac{b_n}{2}\right]\big)}{3} + \frac{\ln(3/2)}{3} \omega(b_n) \right)
$$
 where $\omega(b_n) = 1$ if $b_n$ is even, and $0$ if $b_n$ is odd. 
 Therefore, as $a_n$ and $b_n$ tend to infinity, $I_{n} \sim \frac{\ln(3)}{2} \frac{b_n}{a_n+b_n+1}$. 
 
 When $b_n/a_n$ tends to $0$, $\lim_{n\to\infty} I_{n} = 0$, which gives an expected logarithmically convergent sequence of posteriors. However, taking $a_n=b_n$ leads to $\lim_{n\to\infty} I_{n} = \frac{\ln(3)}{4}$, and the sequence of corresponding posteriors is  not expected logarithmically convergent. 
 
 This example shows that, at least for some statistical models $p(x|\theta)$ and improper prior $\pi(\theta)$,  the notion of expected logarithmic convergence may  depend on the choice of the  sequence of compact sets. It could be interesting to characterizes situations  where it does not. This is left for future works.  
 

   \section{Conclusion and perspectives}
 In this paper, we have shown that 
 the characteristics
of a sequence of proper distributions
given by its FAP or improper limits are quite different.
 As a consequence, there is no clear link between improper distributions and FAPs: a diffuse FAP cannot be considered as the counterpart of some improper distribution.  
 
 In Bayesian statistics, improper distributions are commonly used in practice, even if some paradoxes may occur. They are easy to interpret either through their densities or as conditional probabilities \citep{taralind2016}. At the opposite, diffuse FAPs are never used in practice, mainly because  their constructions are always   implicit and because diffuse FAPs   give information only on the boundary of the parameter space, which is difficult to construe. 
  
 However, FAPs may provide a better understanding of  some limit behavior  that are not captured by   improper distributions.  The fact that our main results rely on explicit constructions of proper prior sequences may be useful to provide counterexamples. For example, in Section \ref{section.elc}, we have shown that the notion of expected logarithmic convergence may depend on the sequence of compact sets.  We hope to use our  results in future works to have a better understanding of  some paradoxical phenomena in Bayesian statistics, such as strong inconsistency or the marginalization paradox.

 \section*{Appendix A}
 We  establish some lemmas useful to prove Theorem \ref{thm.sameqvnotFAP}. The first one is straightforward.
 
 \begin{lemma}
 	\label{lemma.FAPlimitcombconvexe}
 	Let $\{\pi_n^{(1)}\}_n$ and $ \{\pi_n^{(2)}\}_n$ be two sequences of proper distributions and $0\leq \gamma_n\leq 1$ be a sequence of scalars that converges to $0$. Then, the sequence defined by  $\widetilde\pi_n=\gamma_n\pi_n^{(1)}+(1-\gamma_n)\pi_n^{(2)}$ has the same FAP limit points as $\{\pi_n^{(2)}\}_n$.
 \end{lemma}
 
 \begin{proof}
 	For any $f_1,...,f_p\in \mathcal F_b$,  then $(\pi_{n_k}^{(2)}(f_1),...,\pi_{n_k}^{(2)}(f_p) )$ converges to $(\pi(f_1),...,\pi(f_p))$ iff $(\widetilde\pi_{n_k}(f_1),...,\widetilde\pi_{n_k}(f_p) )$ converges to $(\pi(f_1),...,\pi(f_p))$. The result follows. \qed

 \end{proof}
 
 \begin{lemma}
 	\label{lemma.FAPKNslow}
 	Let $\{\pi_n\}_n$ be a sequence of proper priors and  $K_n$ be a non-decreasing sequence of compact sets such that  $\lim_n \pi_n(K_n)=0$, then the sequence defined by  $\widetilde\pi_n=\frac{1}{\pi_n(K_n^c)}\mathds{1}_{K_n^c}\pi_n$ has the same FAP limit points as $\{\pi_n\}_n$.
 \end{lemma}
 
 \begin{proof}
 	
 	First, note that $\{\pi_n\}_n$ is not defined when $\pi_n(K_n)=1$,  but this cannot occur more than a finite number of times.
 	For any $f\in \mathcal F_b$, $\pi_n(f)=$ $\mathds{1}_{K_n}\pi_n(f)+\mathds{1}_{K_n^c}\pi_n(f)=$ $\pi_n(\mathds{1}_{K_n}f)+\pi_n(K_n^c)\widetilde\pi_n(f)$. 
 	Since $f$ is bounded, $\lim_n\pi_n(\mathds{1}_{K_n}f)$ $=$ $0$.  Moreover, $\lim_n(K_n^c)=1$. Thus, for any $f_1,...,f_p\in \mathcal F_b$,   $(\pi_{n_k}(f_1),...,\pi_{n_k} (f_p) )$ converges to $(\pi(f_1),...,\pi(f_p))$ iff $(\widetilde\pi_{n_k}(f_1),...,\widetilde\pi_{n_k}(f_p) )$ converges to $(\pi(f_1),...,\pi(f_p))$.\qed
   	
 \end{proof}
 
 At the opposite of Lemma \ref{lemma.FAPKNslow}, the following lemma shows that if we consider the restriction of a sequence $\{\pi_n\}_n$ of a proper or improper distribution on a exhaustive increasing sequence $\{K_n\}_n$ of compact sets, we preserve the q-vague limits.
 \begin{lemma}\label{lemma.cvqvtronque}
 	Let $K_n$ be a non-decreasing sequence of compact sets such that $\cup_n K_n=\Theta$ and such that, for any compact $K$, there exists $N$ such that $K\subset K_N$. A sequence   $\{\pi_n\}_n$  of Radon measures converges q-vaguely  to the Radon measure $\pi$ if and only if $\widetilde \pi_n=\frac{1}{\pi_n(K_n)} \mathds{1}_{K_n}\pi_n$ converges q-vaguely to $\pi$.
 \end{lemma}
 \begin{proof}
 	Assume that  $\pi_n$ converges q-vaguely to $\pi$, then there exists some positive scalars $\{a_n\}_n$ such that for any $f$ in $\mathcal C_K$, $\lim_n a_n\pi_n(f)=\pi(f)$. Put $\widetilde a_n= a_n\,\pi_n(K_n)$ and denote by $K_f$ a compact set that includes the support of $f$. Then, there exists an integer $N$ such that $K_f\subset K_n$ for $n>N$.
 	Therefore, for $n>N$,  $\widetilde a_n\widetilde \pi_n(f)=a_n\pi_n(f)$. The result and its reciprocal follow.\qed 
 \end{proof}

  \section*{Appendix B}
 
 We prove here Proposition \ref{prop.PoissonlimitFAP} of section
 \ref{section.poisson}.
 
 In order to show that  $\pi$ is SI-uniform,  we consider   $\pi_n$ as
 a distribution  on the set of positive and negative integers, extending it by $0$ on the non-positive
 integers. Define by $\pi^{(k)}_n$  the shifted distribution: $\pi^{(k)}_n (A) = \pi_n
 (A+k)$, for any subset $A$ of the set of integers.  One knows that
 $\|\pi^{(k)}_n - \pi_n \|_{TV} \leq \frac{k}{\sqrt{2\pi n}}$, where
 $\|\cdot\|_{TV}$ is the total variation norm. Therefore, for any subset
 of $\mathbb N$, $\lim_{n\to\infty} |\pi_n(A+k) - \pi_n(A) | =0$. Letting
 $n$ go to infinity, we deduce that, for any FAP limit point $ \pi$ of $\pi_n$,
 and any integer $k$ : $ \pi (A+k) =  \pi (A)$.

 The fact that $\pi$ is BS-uniform comes from an easy
 adaptation of the Hoeffding inequality in that context. Let $(X_k)_{k
 	\in \mathbb N}$ be a Bernoulli scheme, of parameter $p$, and denote by
 $\mathbb P$ the associated probability.
 Hoeffding inequality gives, that, for any $n$~:
 
 $$
 \mathbb P
 \left\{
 \bigg|
 \sum_{k=0}^{\infty} e^{-n}\frac{k^n}{n!} (X_k(\omega) -p)
 \bigg|
 \geq
 t
 \right\}
 \; \leq \;
 2 e^{-2 c \sqrt{2\pi n}\, t^2},
 $$
 for some positive constant $c$. The expected conclusion is then
 obtained thanks to the  Borel-Cantelli lemma.
 
 The fact that some of the  limit points $\pi$ of $\{\pi_n\}_n$ are not LRF uniform is a
 direct  consequence of the following lemma.

 \begin{lemma}
 	For any $0\leq p,p'\leq 1$, there exists a set $A$ and some FAP limit points  ${\pi}$ of  $\{{\pi}_n\}_n$  such that $LRF(A)=p$ and ${\pi}(A)=p'$.
 \end{lemma}
 \begin{proof}
 	First note  that, for any set $A'$, $LRF(A')=p$ if, and only
 	if, $\sharp \{ k \leq n, \, k\in A'\} =$ $ pn + o(n)$. Therefore, for any
 	set $A$ with $LRF(A)=p$ and for any set $B$ such that $\sharp \{k\leq n
 	, \, k\in B\} = o(N)$, one has both $LRF(A\cup B) = p$ and
 	$LRF(A\setminus B ) = p$. Take now for set $B$ the following~:
 	$$
 	B = \bigcup_{k\in\mathbb N} \big\{ u\in\mathbb N \; :\; 4^k -2^k k
 	\leq u \leq  4^k+2^k k \big\}.
 	$$
 	For that $B$, one has~:
 	$$
 	\limsup_{n\to\infty} \frac{\sharp\{k\leq n, \, k \in B\}}{n+1} =
 	\lim_{k \to\infty} \frac{\sum_{i=0}^{k} 2^{i+1} i }{4^k + 2^k k}
 	\; \leq \; \lim_{k \to\infty}  \frac{(k+1)2^{k+2}}{4^k } = 0,
 	$$
 	and thus $LRF(B) =0$. However, $\pi_{4^k} (B)$ converges to $1$.
 	Indeed, if $U_k$ is some random variable with law $\pi_{4^k}$, one has~:
 	$$
 	\pi_{4^k}
 	\big(
 	\big\{ u\in\mathbb N \; :\; 4^k -2^k k \leq u \leq  4^k+2^k k \big\}
 	\big)
 	=
 	\mathbb P
 	\bigg( \frac{U_k -4^k}{\sqrt{4^k}} \in
 	\big[ - k \, ; \, k
 	\big]
 	\bigg).
 	$$
 	The right-hand side term above converges to $1$ thanks to the
 	central limit theorem. Hence $LRF(A\cup B) = LRF(A\setminus B) = p$
 	while $\pi_{4^k}(A\cup B)$ converges to $1$, and $\pi_{4^k}(A\setminus
 	B) $ converges to $0$. Now, for any $p'\in[0;1]$, choose two numbers
 	$a<b$, so that
 	$
 	p'=\int_a^b \frac{e^{-u^2/2}}{\sqrt{2\pi}} du $. Take the set $B'$
 	to be~:
 	$$
 	B' = \bigcup_{k\in\mathbb N}
 	\big\{
 	u\in\mathbb N \; :\;
 	4^k +2^k \text{max}(-k,a) \leq u \leq  4^k+2^k \text{min}(k,b)
 	\big\},
 	$$
 	then $LRF(B')=0$ again and $\pi_{4^k} (B')$ converges to $p'$
 	, still thanks to the central limit theorem. Let $A = (A'\setminus
 	B)\cup B'$. Then $LRF (A) =p $ and
 	$\lim_{k \to \infty} \pi_{4^k} (A) =p'$. Now, any FAP limit point $\pi$
 	of subsequence $\{\pi_{4^k}\}_k$ is also a FAP limit point of $\{\pi_k\}_k$. Hence, $\pi$ is
 	 SI-uniform and BS-uniform,
 	but one has $ \pi (A)= p'$ and $LRF(A)=p$.\qed
 \end{proof}




\end{document}